\documentclass[12pt]{article}
\usepackage{graphicx}
\usepackage{amsmath,amsthm,amssymb,enumerate}
\usepackage{euscript,mathrsfs}
\usepackage[left=2cm,right=2cm,top=3.5cm,bottom=3.5cm]{geometry}
\usepackage{color}

\newtheorem{theorem}{Theorem}[section]
\newtheorem{lemma}[theorem]{Lemma}
\newtheorem{corollary}[theorem]{Corollary}

\theoremstyle{definition}
\newtheorem{definition}[theorem]{Definition}

\newtheorem{remark}[theorem]{Remark}

\newcommand{\bfphi}{\boldsymbol{\varphi}}
\newcommand{\Ov}[1]{\overline{#1}}
\newcommand{\DC}{C^\infty_c}
\newcommand{\vr}{\varrho}
\newcommand{\vt}{\vartheta}
\newcommand{\vc}[1]{{\bf #1}}
\newcommand{\vu}{\vc{u}}
\newcommand{\vm}{\vc{m}}
\newcommand{\Div}{{\rm div}_x}
\newcommand{\Grad}{\nabla_x}
\newcommand{\Del}{\Delta_x}
\newcommand{\dx}{\ {\rm d} {x}}
\newcommand{\dt}{\ {\rm d} t }
\newcommand{\dxdt}{\dx \dt}
\newcommand{\intO}[1]{\int_{\Omega} #1\dx}

\begin{document}
	
	\title{On the density of ``wild'' initial data for the compressible Euler system}

	\author{Eduard Feireisl \and Christian Klingenberg \and Simon Markfelder}
	
	\maketitle
	
\bigskip

\centerline{Institute of Mathematics of the Academy of Sciences of the Czech Republic}

\centerline{\v Zitn\' a 25, CZ-115 67 Praha 1, Czech Republic}

\centerline{and}

\centerline{Institute of Mathematics, TU Berlin}

\centerline{Strasse des 17.Juni, Berlin}

\bigskip

\centerline{Department of Mathematics, W\"urzburg University}

\centerline{Emil-Fischer-Str. 40, 97074 W\"urzburg, Germany}

\bigskip

	\begin{abstract}
		We consider a class of ``wild'' initial data to the compressible Euler system that
		give rise to infinitely many admissible weak solutions via the method of convex integration.
		We identify the closure of this class in the natural $L^1-$topology and show that its complement is rather
		large, specifically it is an open dense set. 
	\end{abstract}


	
	\section{Introduction}
	\label{E}
	
	The motion of a fluid in gas dynamics is described by means of the standard variables: the fluid density $\vr = \vr(t,x)$,
	the macroscopic velocity $\vu = \vu(t,x)$, and the absolute temperature $\vt = \vt(t,x)$ satisfying the \emph{Euler system}:

	\begin{equation} \label{E1}
	\begin{split}
	\partial_t \vr + \Div (\vr \vu) &= 0,\\
	\partial_t (\vr \vu) + \Div (\vr \vu \otimes \vu) + \Grad p(\vr, \vt) &= 0,\\
	\partial_t \left( \frac{1}{2} \vr |\vu|^2 + \vr e(\vr, \vt) \right)
	+ \Div \left[ \left( \frac{1}{2} \vr |\vu|^2 + \vr e(\vr, \vt) \right) \vu \right]
	+ \Div (p(\vr, \vt) \vu )&=0.
	\end{split}
	\end{equation}

	\noindent
	We suppose  the fluid is confined to a smooth bounded domain $\Omega \subset R^N$, $N=2,3$, with
	the impermeable boundary:

	\begin{equation} \label{E2}
	\vu \cdot \vc{n}|_{\partial \Omega} = 0.
	\end{equation}

	The state of the system at the reference time $t = 0$ is given by the initial conditions:

	\begin{equation} \label{E3}
	\vr(0, \cdot) = \vr_0, \ \vu(0, \cdot) = \vu_0, \ \vt(0, \cdot) = \vt_0.
	\end{equation}

	\subsection{Thermodynamics}
	
	The system \eqref{E1} contains two thermodynamics functions: the presure $p = p(\vr, \vt)$ and the internal energy
	$e = e(\vr, \vt)$. In accordance with the Second law, we postulate the existence of the entropy $s = s(\vr, \vt)$ related to $p$ and $e$
	via Gibbs' equation
	\[
	\vt Ds = De + D \left( \frac{1}{\vr} \right) p.
	\]
	If the functions $\vr$, $\vu$, and $\vt$ are continuously differentiable, the relations \eqref{E1} give rise to the entropy balance:
	\begin{equation} \label{E4}
	\partial_t (\vr s(\vr,\vt)) + \Div (\vr s(\vr, \vt) \vu) = 0.
	\end{equation}
	In the context of weak solutions considered in this paper, it is customary to relax \eqref{E4} to \emph{inequality}
	\begin{equation} \label{E5}
	\partial_t (\vr s(\vr,\vt)) + \Div (\vr s(\vr, \vt) \vu) \geq 0.
	\end{equation}
	
	The Euler system \eqref{E1} can be written in the conservative variables:
	the density $\vr$, the momentum $\vc{m} = \vr \vu$, and the energy $E = \frac{1}{2} \vr |\vu|^2 + \vr e(\vr, \vt)$:
	\begin{equation} \label{E6}
	\begin{split}
	\partial_t \vr + \Div \vc{m} &= 0,\\
	\partial_t \vc{m} + \Div \left( \frac{\vc{m} \otimes \vc{m} }{\vr} \right) + \Grad p(\vr, \vc{m}, E) &= 0,\\
	\partial_t E
	+ \Div \left( E \frac{\vm}{\vr} \right)
	+ \Div \left (p(\vr, \vm, E) \frac{\vm}{\vr} \right)&=0.
	\end{split}
	\end{equation}
	Although we divide by $\vr$ when passing to the formulation in conservative variables, we can still treat vacuum by requiring the momentum to be zero in vacuum zones, see Section \ref{DMV}.
	
	The forthcoming analysis leans essentially on the \emph{hypothesis of thermodynamics stability}. The latter can be formulated either
	in the standard variables:
	\begin{equation} \label{E7}
	\frac{\partial p(\vr, \vt)}{\partial \vr} > 0, \ \frac{\partial e(\vr, \vt)}{\partial \vt} > 0;
	\end{equation}
	or in the conservative variables:
	\begin{equation} \label{E8}
	\begin{split}
	(\vr, \vc{m}, E) \mapsto &S(\vr, \vc{m}, E) \equiv \vr s (\vr, \vc{m}, E) \\
	&\ \mbox{is a strictly concave function of}\ (\vr, \vm, E).
	\end{split}
	\end{equation}
	
	\subsection{Smooth and weak solutions}
	
	Smooth solutions are at least continuously differentiable on the set $[0,T) \times \Ov{\Omega}$,
	$\vr > 0$, $\vt > 0$, and satisfy
	(\ref{E1}--\ref{E3}) pointwise. Smooth solutions are known to exist locally on a maximal interval $[0,T_{\rm max})$
	as long as the initial data are smooth enough and satisfy the corresponding compatibility conditions on $\partial \Omega$, see last part of Section 4 for details.
	Moreover, it is well known that $T_{\rm max} < \infty$ for a ``generic'' class of data.
	
	The weak solutions satisfy \eqref{E1}, \eqref{E2} in the sense of distributions, the conservative variables are weakly continuous in time so that
	the initial conditions are well defined. The weak solutions are called \emph{admissible} if the entropy inequality
	\eqref{E5} holds in the weak sense. The existence of global in time weak (admissible) solutions for given initial data is an open problem. Recently, however, the method of convex integration developed in the context of the incompressible Euler system
	by De Lellis and Sz\' ekelyhidi \cite{DelSze10} has been adapted to identify a class of initial data for which the problem (\ref{E1}--\ref{E3}) admits \emph{infinitely many} admissible weak solutions defined on a given time interval $(0,T)$, \cite{FKKM17}. Similar results have been obtained also for the associated Riemann problem in \cite{BKKMM18}.
	
	The initial data for which the problem admits local smooth solution will be termed \emph{regular}, the data giving rise
	to infinitely many admissible weak solutions are called \emph{wild}. Our goal is to identify the class of regular data that can be obtained as limits of wild data. Note that for the \emph{incompressible} Euler system, the wild data (velocities) are dense in the Lebesgue space $L^2(\Omega; R^N)$, see Sz\' ekelyhidi and Wiedemann \cite{SzeWie12}.
	
	In the present paper, we restrict ourselves to the class of wild solutions resulting from a rather general splitting method which goes back to \cite{Feireisl16}. To the best of our knowledge, all convex integration results for the Euler system \eqref{E6} in the literature use such a splitting method.
	
	A vector field $\vm$ can be decomposed by means of Helmholtz decomposition, i. e. it can be written as the sum of a solenoidal vector field $\vc{H}[\vm]$ and the gradient of a scalar field.
	
	More precisely, the Helmholtz projection operator $\vc{H}$ is defined as 
	\begin{equation*}
	\begin{split}
	&\vc{m} = \vc{H}[ \vc{m} ] + \vc{H}^\perp [ \vc{m} ], \\&\mbox{where}\ \vc{H}^\perp = \Grad \Phi,\
	\Del \Phi = \Div \vc{m},\ (\Grad \Phi - \vc{m}) \cdot \vc{n}|_{\partial \Omega} = 0,\ \intO{ \Phi } = 0.
	\end{split}
	\end{equation*}
	
	Accordingly, for $\vc{m} = \vc{v} + \Grad \Phi$, $\vc{v} = \vc{H} [\vc{m}]$, the system \eqref{E6} can be written in the form
	\begin{equation} \label{NE6}
	\begin{split}
	\partial_t \vr + \Del \Phi &= 0,\\
	\partial_t \vc{v} + \vc{H} \left[ \Div \left( \frac{\vc{m} \otimes \vc{m} }{\vr} - \frac{1}{N} \frac{|\vc{m}|^2}{\vr} \mathbb{I} \right) \right]  &= 0,\\
	\partial_t (\Grad \Phi) + \vc{H}^\perp  \left[ \Div \left( \frac{\vc{m} \otimes \vc{m} }{\vr}  - \frac{1}{N} \frac{|\vc{m}|^2}{\vr} \mathbb{I}\right)\right] +\qquad\qquad\qquad &
	\\
	+ \Grad \left( \frac{1}{N} \frac{|\vc{m}|^2}{\vr} \right) + \Grad p(\vr, \vc{m}, E)  &=0,\\
	\partial_t E
	+ \Div \left( E \frac{\vm}{\vr} \right)
	+ \Div \left (p(\vr, \vm, E) \frac{\vm}{\vr} \right)&=0.
	\end{split}
	\end{equation}
	
	To the best of our knowledge, there are only two convex integration an\-satzes for obtaining wild solutions to the Euler system \eqref{E6} in the literature. Let us begin with the first one, which is used in \cite{FKKM17}. This ansatz is based on replacing equations $\eqref{NE6}_2$, $\eqref{NE6}_3$ by the system
	\begin{equation} \label{NE7}
	\begin{split}
	\partial_t \vc{v} +\Div \left( \frac{\vc{m} \otimes \vc{m} }{\vr} - \frac{1}{N} \frac{|\vc{m}|^2}{\vr} \mathbb{I} \right)   &= 0, \ \Big( \Div \vc{v} = 0 \Big), \\
	\partial_t (\Grad \Phi)
	+ \Grad \left( \frac{1}{N} \frac{|\vc{m}|^2}{\vr} \right) + \Grad p(\vr, \vc{m}, E)  &=0,\ \Big( \partial_t \vr + \Del \Phi = 0 \Big).
	\end{split}
	\end{equation} 
	Equation $\eqref{NE7}_1$ represents the volume preserving part of the motion, while
	$\eqref{NE7}_2$ may be seen as a wave equation governing the propagation of acoustic waves.
	Note that the fields $\vc{v}$ and $\Grad \Phi$ as well as the fluxes in \eqref{NE7} are orthogonal with respect to the $L^2$ scalar product, specifically, 
	\[
	\intO{ \vc{v} \cdot \Grad \Phi } = 0,\ \left[ \frac{\vc{m} \otimes \vc{m} }{\vr} - \frac{1}{N} \frac{|\vc{m}|^2}{\vr} \mathbb{I} \right]
	: \left[ \left( \frac{1}{N} \frac{|\vc{m}|^2}{\vr} + p(\vr, \vc{m}, E) \right) \mathbb{I} \right]  = 0,
	\]
	where the latter is true since it is a product of a traceless matrix and a multiple of the identity matrix.
	
	Although $\eqref{NE7}_1$ is apparently overdetermined, it admits (infinitely many) weak solutions for any fixed $\vr$ and $\Phi$, cf.  Appendix \ref{A}.
	
	\begin{definition} \label{ND1}
		
		A weak solution $[\vr, \vm, E]$ of the Euler system \eqref{E6} is called \emph{wild solution} if $[\vr, \vm, E]$ satisfy \eqref{NE7}, with
		$\vc{v} = \vc{H}[\vc{m}]$, $\Grad \Phi = \vc{H}^\perp [\vm]$. The corresponding initial data $[\vr_0, \vm_0, E_0]$ are called \emph{wild initial data}.
		
	\end{definition}
	
	As indicated above, Theorem \ref{TS1} yields solutions which are wild in the sense of Definition \ref{ND1}. This guarantees that the set of wild solutions, as well as the set of wild initial data are non-empty. However Definition \ref{ND1} is more general, which means that there might be wild solutions (and hence wild initial data) in the sense of Definition \ref{ND1} that cannot be obtained by Theorem \ref{TS1}.
	
	Note that a technique similar to \eqref{NE7} has been used also for the simplified isentropic Euler system 
	by Chiodaroli \cite{Chiodaroli14}. 
	
	The second convex integration ansatz, that is available in the literature, is based on the analysis of the corresponding Riemann 
	problem, see \cite{BKKMM18}, \cite{ChiKre14}, \cite{KliMar18} among others. We will first focus on the ansatz \eqref{NE7} and afterwards extend our results to the 
	wild solutions obtained via the Riemann problem in Section \ref{DD}.

	\subsection{Main result}
	
	We are ready to formulate our main result.
	To avoid the situation when the temperature approaches absolute zero, we restrict ourselves to the phase space
	\begin{equation*}
	\begin{split}
	&\texttt{L}^1_{+,s_0} (\Omega; R^{N + 2}) \\
	&= \left\{ [\vr, \vm, E] \in L^1(\Omega; R^{N+2}) \ \Big|
	\ \vr \geq 0, \ E \geq 0, \ s(\vr, \vm, E) \geq s_0 > - \infty \right\}.
	\end{split}
	\end{equation*}
	Note that this is not very restrictive as any admissible weak solution satisfies the minimum principle
	\[
	s(\vr, \vm, E)(t,x) \geq {\rm ess}\inf_{y \in \Omega} s(\vr_0(y), \vm_0(y), E_0(y)) \ \quad\mbox{a.a. in}\ (0,T) \times \Omega,
	\]
	cf. \cite[Section 2.1.1]{BreFei17}.
	
	We say that a sequence of data $\left\{ \vr_{0,n}, \vu_{0,n}, \vt_{0,n} \right\}_{n=1}^\infty$, or, equivalently, \\$\left\{ \vr_{0,n}, \vm_{0,n}, E_{0,n} \right\}_{n=1}^\infty$, \emph{(W)-converges} to
	$[\vr_0, \vm_0, E_0]$,
	\[
	[ \vr_{0,n}, \vm_{0,n}, E_{0,n} ] \stackrel{(W)}{\rightarrow} [\vr_0, \vm_0, E_0],
	\]
	if
	\begin{itemize}
		\item
		\[
		\vr_{0,n} > 0,\quad\ s(\vr_{0,n}, \vm_{0,n}, E_{0,n}) \geq  s_0 > -\infty;
		\]
		\item
		\begin{equation} \label{E9}
		[ \vr_{0,n}, \vm_{0,n}, E_{0,n} ] \to [\vr_0, \vm_0, E_0] \quad\ \mbox{in}\ L^1(\Omega; R^{N+2});
		\end{equation}
		\item
		the initial data $[\vr_{0,n},
		\vm_{0,n}, E_{0,n}]$ give rise to a sequence of admissible weak solutions $[\vr_n, \vm_n, E_n]$ satisfying
		\begin{equation} \label{E10}
		\begin{split}
		\int_0^T \intO{ \left( \frac{\vm_n \otimes \vm_n} {\vr_n} - \frac{1}{N} \frac{ |\vm_n|^2}{\vr_n} \mathbb{I} \right)
			: \Grad^2 \varphi } \dt \to 0 \quad\ \mbox{as}\ n \to \infty\ \\\mbox{for any}\ \varphi \in \DC((0,T) \times \Omega).
		\end{split}
		\end{equation}

	\end{itemize}
	
	Note that, in view of \eqref{NE7},
	\[
	\vc{H}^\perp \left[ \Div \left( \frac{\vm_n \otimes \vm_n} {\vr_n} - \frac{1}{N} \frac{ |\vm_n|^2}{\vr_n} \mathbb{I} \right) \right] = 0
	\]
	whenever $[\vr_n, \vm_n, E_n]$ are wild solutions in the sense of Definition \ref{ND1}. In particular, 
	\eqref{E10} is trivially satisfied for any sequence of wild initial data in the sense of Definition \ref{ND1}.
	
	Finally, note that the sequence of solutions $\left\{ \vr_{n}, \vm_{n}, E_{n} \right\}_{n=1}^\infty$ need not {\it a priori} converge strongly, and, consequently, \eqref{E10} does not impose any similar restriction on a possible limit solution emanating from $[\vr_0, \vm_0, E_0]$.
	
	\begin{definition} \label{ND2}
		
		We say that a trio $[\vr_0, \vm_0, E_0]$ is \emph{reachable} if there exists a sequence of initial data
		$\left\{ \vr_{0,n}, \vm_{0,n}, E_{0,n} \right\}_{n=1}^\infty$ such that
		\[
		[ \vr_{0,n}, \vm_{0,n}, E_{0,n} ] \stackrel{(W)}{\rightarrow}  [\vr_0, \vm_0, E_0].
		\]
		
	\end{definition}
	
	Our goal is to show that the set of reachable data is ``small'' in the sense that its complement is an open dense set in $\texttt{L}^1_{+,s_0} (\Omega; R^{N+2})$. For the sake of simplicity,
	we focus on the equations of gas dynamics, where the pressure is given by Boyle--Mariotte law
	\[
	p(\vr, \vt) = \vr \vt, \quad\ \ e(\vr, \vt) = c_v \vt, \ c_v > 0.
	\]
	Equivalently, in the conservative variables,
	\begin{equation} \label{E11}
	p = \frac{1}{c_v} \left( E - \frac{1}{2} \frac{|\vc{m}|^2}{\vr} \right),\ \quad\vt = \frac{1}{c_v \vr } \left( E - \frac{1}{2} \frac{|\vc{m}|^2}{\vr} \right) > 0.
	\end{equation}
	Here the kinetic energy is defined as
	\[
	\frac{|\vc{m}|^2}{\vr} =  \left\{ \begin{array}{ll} 0 \ &\mbox{if}\ \vm = 0 \\
	\infty \ &\mbox{if}\ \vr = 0 \\
	\frac{|\vc{m}|^2}{\vr} \ &\mbox{otherwise.}
	\end{array} \right.
	\]
	The entropy reads
	\[
	S (\vr, \vm, E) = \vr c_v \log \left( \frac{
		E - \frac{1}{2} \frac{|\vc{m}|^2}{\vr}  }{ c_v \vr^{1 + \frac{1}{c_v}}}.
	\right)
	\]

	Here is our main result.

	\begin{theorem} \label{TE1}
		Let $s_0 \in R$ be given.
		Let $\Omega \subset R^N$, $N=2,3$ be a bounded smooth domain. \\
		Then the complement of the set of reachable data is an open dense set in
		$\texttt{L}^1_{+,s_0} (\Omega; R^{N + 2})$.
	\end{theorem}

	\begin{corollary} \label{CE1} 
		
		The complement of the set of wild initial data (in the sense of Definition \ref{ND1}) contains an open dense set in
		$\texttt{L}^1_{+,s_0} (\Omega; R^{N + 2})$. 
	\end{corollary}
	
	\begin{proof}
		As mentioned above, wild initial data are reachable and hence the complement of the set of wild initial data contains the complement of the set of reachable data. Thus Theorem \ref{TE1} yields the claim.
	\end{proof}
	
	The main part of the paper is devoted to the proof of Theorem \ref{TE1}. Our strategy is to identify a large (dense) set of \emph{regular}
	data for the Euler system that are not reachable. To see this, we consider initial data $[\vr_0, \vm_0, E_0]$ giving rise to a smooth
	solution $[\vr, \vm, E]$ defined on a maximal time interval $[0, T_{{\rm max}})$. Then we proceed in several steps:
	
	\begin{itemize}
		\item Assuming the data $[\vr_0, \vm_0, E_0]$ are reachable we show that the associated sequence of
		solutions $\{ \vr_n, \vm_n, E_n \}_{n=1}^\infty$ generates a Young measure that can be identified with a dissipative measure
		valued (DMV) solution to the Euler system in the sense of \cite{BreFei17}, see Section \ref{DMV}.
		
		\item Next we realize that, thanks to the strong convergence required in \eqref{E9}, the limit DMV solution starts from
		the initial data represented by the parameterized family of Dirac masses,
		\[
		\left\{ \delta_{\vr_0(x),\vm_0(x), E_0(x) } \right\}_{x \in \Omega}.
		\]
		Thanks to the general \emph{weak--strong uniqueness} principle, we conclude that the DMV solution coincides with the strong
		solution on the time interval $[0, T_{\rm max})$, more specifically, the DMV solution is represented by the parameterized family
		of Dirac masses,
		\[
		\left\{ \delta_{\vr(t,x),\vm(t,x), E(t,x) } \right\}_{(t,x) \in (0, T_{{\rm max}}) \times \Omega},
		\]
		see Section \ref{WS}. In particular, the solutions $[\vr_n, \vm_n, E_n]$ converge strongly and we may assume
		\begin{equation} \label{E12}
		\vr_n \to \vr, \ \vm_n \to \vm,\ E_n \to E \ \quad\mbox{a.a. in}\ (0,T_{{\rm max}}) \times \Omega.
		\end{equation}
		
		\item Thanks to the strong convergence \eqref{E12}, condition \eqref{E10}, and smoothness of the limit solution, we deduce that
		\[
		\Div \Div \left( \frac{ \vc{m} \otimes \vc{m} }{\vr} \right) - \frac{1}{N} \Del \left( \frac{|\vc{m}|^2}{\vr} \right) = 0 \ \quad\mbox{in}\ (0,T) \times \Omega.
		\]
		see Section \ref{WS} for details. In particular, as the limit solution is continuous up to $t= 0$,
		\begin{equation} \label{E13}
		\Div \Div \left( \frac{ \vc{m}_0 \otimes \vc{m}_0 }{\vr_0} \right) - \frac{1}{N} \Del \left( \frac{|\vc{m}_0|^2}{\vr_0} \right) = 0 \ \quad\mbox{in}\ \Omega.
		\end{equation}

		\item In Section \ref{FI}, we show that \eqref{E13} can be satisfied for a very narrow class of the initial data. In particular, we complete the proof of Theorem
		\ref{TE1}.
		
	\end{itemize}
	A related result for the wild solutions obtained via the Riemann problem is shown in Section \ref{DD}.
	
	\section{Dissipative measure valued (DMV) solutions}
	\label{DMV}
	
	\subsection{Definitions}
	
	First we recall the standard definition of the admissible weak solutions to the Euler system:
	
	\begin{definition}[admissible weak solution] \label{DDMV2}
		
		Let $\gamma = 1 + \frac{1}{c_v}$.
		We say that $[\vr, \vm, E]$ is an \emph{admissible weak solution} to the Euler system (\ref{E1}--\ref{E3}) if
		\begin{align*}
		\vr &\in C_{{\rm weak}}([0,T]; L^\gamma (\Omega)), \ \vr \geq 0,\ \\
		\vm &\in C_{{\rm weak}}([0,T]; L^{\frac{2 \gamma}{\gamma + 1}}(\Omega; R^N)), \ \vm=0\ \text{whenever}\ \vr=0,\\
		E &\in C_{{\rm weak}} ([0,T]; L^1(\Omega)), \ E\geq 0,
		\end{align*}
		and the following holds:
		\begin{itemize}
			
			\item
			\[
			\int_0^T \intO{ \left[ \vr  \partial_t \varphi + \vm  \cdot \Grad \varphi \right] } \dt = - \intO{ \vr_0 \varphi(0, \cdot) }
			\]
			for any $\varphi \in C^1_c([0,T) \times \Ov{\Omega})$;
			\item
			\begin{equation*}
			\begin{split}
			\int_0^T \intO{ \left[ \vm \cdot \partial_t\bfphi +
				\left( \frac{\vm \otimes \vm}{\vr} \right)  : \Grad \bfphi + p
				\left(\vr, \vc{m}, E \right) \Div \bfphi \right] } \dt \quad\\
			= - \intO{ \vc{m}_0  \cdot \bfphi (0, \cdot) }
			\end{split}
			\end{equation*}
			for any $\bfphi \in C^1_c([0,T) \times \Ov{\Omega}; R^N)$, $\bfphi \cdot \vc{n}|_{\partial \Omega} = 0$;
			
			\item
			\begin{equation} \label{DMV1}
			\int_0^T \intO{ \left[ E \partial_t \varphi + \left( E + p(\vr, \vm, E) \right) \left( \frac{\vm}{\vr} \right) \cdot
				\Grad \varphi \right] } \dt = - \intO{ E_0 \varphi (0, \cdot) }
			\end{equation}
			for any $\varphi \in C^1_c ([0,T) \times \Ov{\Omega})$;

			\item
			\begin{equation} \label{RenEntr}
			\begin{split}
			- &\intO{ \vr_0 Z \left( s  \left(\varrho_0, \vc{m}_0, E_0  \right) \right) \varphi(0, \cdot) }
			\\ &\geq \int_0^T \intO{ \left[ \vr Z \left( s \left(\vr, \vc{m}, E \right) \right)  \partial_t
				\varphi  + Z \left( s \left(\varrho, \vm, E \right) \right) \vc{m} \cdot \Grad \varphi  \right] } \dt
			\end{split}
			\end{equation}
			for any $\varphi \in C^1_c([0,T) \times \Ov{\Omega})$, $\varphi \geq 0$,
			and $Z(s) \leq \Ov{Z}$, $Z'(s) \geq 0$.

		\end{itemize}

	\end{definition}
	
	\begin{remark}
		Here the entropy inequality is satisfied in the renormalized sense, see e. g. \cite{CheFri02}. 
	\end{remark}

	To carry out the programme delineated at the end of the preceding section, we introduce the concept of
	dissipative measure valued (DMV) solution, see \cite{BreFei17}.
	Let
	\[
	\mathcal{P} = \left\{ [\vr, \vm, E] \ \Big| \ \vr \geq 0, \ \vm \in R^N, \ E \geq 0 \right\}
	\]
	be the phase space associated to the Euler system.
	
	\begin{definition} [DMV solution] \label{DDMV1}
		
		A parametrized family of probability measures $\{ \mathcal{V}_{t,x} \}_{(t,x) \in (0,T) \times \Omega}$
		on the set $\mathcal{P}$,
		\[
		\mathcal{V} \in L^\infty_{\rm weak-(*)} ((0,T) \times \Omega; {\rm prob} [ \mathcal{P} ] )
		\]
		is called a \emph{dissipative measure--valued (DMV) solution} to the compressible Euler system
		with the initial data $\{ \mathcal{V}_{0,x} \}_{x \in \Omega}$,
		if the following holds:
		
		\begin{itemize}
			
			\item
			\[
			\int_0^T \intO{ \left[ \left< \mathcal{V}_{t,x}; \vr \right> \partial_t \varphi + \left< \mathcal{V}_{t,x}; \vm \right> \cdot \Grad \varphi \right] } \dt = - \intO{ \left< \mathcal{V}_{0,x}; \vr \right> \varphi(0, \cdot) }
			\]
			for any $\varphi \in C^1_c([0,T) \times \Ov{\Omega})$;
			\item
			\[
			\begin{split}
			\int_0^T \intO{ \bigg[ \left< \mathcal{V}_{t,x}; \vm \right> \cdot \partial_t\bfphi &+
				\left< \mathcal{V}_{t,x} ; \frac{ \vm \otimes \vm}{\vr} \right> : \Grad \bfphi \\
				&\qquad\qquad+ \left< \mathcal{V}_{t,x}; p
				\left(\vr, \vc{m}, E \right) \right> \Div \bfphi \bigg] } \dt \\ &=
			\int_0^T \int_{\Omega} \Grad \bfphi : {\rm d}\mu_C - \intO{ \left< \mathcal{V}_{0,x}; \vc{m} \right> \cdot \bfphi (0, \cdot) }
			\end{split}
			\]
			for any $\bfphi \in C^1_c([0,T) \times \Ov{\Omega}; R^N)$, $\bfphi \cdot \vc{n}|_{\partial \Omega} = 0$, where $\mu_C$ is a (vectorial) signed measure on $[0,T] \times \Ov{\Omega}$;
			\item
			\begin{equation} \label{DMV2}
			\intO{ \left< \mathcal{V}_{\tau ,x}; E \right> }  + \mathcal{D}(\tau) \leq
			\intO{ \left< \mathcal{V}_{0,x}; E \right> }
			\end{equation}
			for a.a. $\tau \in (0,T)$, where $\mathcal{D} \in L^\infty(0,T)$, $\mathcal{D} \geq 0$ is called
			dissipation defect;

			\item
			\[
			\begin{split}
			&- \intO{  \left< \mathcal{V}_{0,x}; \vr Z \left(s  \left(\varrho, \vc{m}, E  \right) \right)  \right> \varphi(0, \cdot) }
			\\ &\geq \int_0^T \intO{ \left[ \left< \mathcal{V}_{t,x} ; \vr Z \left( s \left(\vr, \vc{m}, E \right) \right) \right> \right] \partial_t
				\varphi} \dt \\&\quad+ \int_0^T \intO{ \left[
				\left< \mathcal{V}_{t,x} ; Z \left( s \left(\varrho, \vm, E \right) \right) \vc{m} \right> \cdot \Grad \varphi  \right] } \dt
			\end{split}
			\]
			for any $\varphi \in C^1_c([0,T) \times \Ov{\Omega})$, $\varphi \geq 0$,
			and $Z(s) \leq \Ov{Z}$, $Z'(s) \geq 0$;
			\item
			
			\[
			\left\| \mu_C \right\|_{\mathcal{M}([0, \tau] \times \Omega) } \leq c \int_0^\tau \mathcal{D}(t) \dt \ \quad\mbox{for a.a.}\ \tau \in (0,T).
			\]
			
		\end{itemize}

	\end{definition}
	
	The reader will have noticed that Definition \ref{DDMV1} is not a mere measure--valued variant of Definition \ref{DDMV2}. In particular, the energy equation \eqref{DMV1} has been replaced by its integrated version \eqref{DMV2}. The reader may consult
	\cite{BreFei17} for a thorough discussion of this new concept of solution. Note in particular, that \eqref{DMV2} implies 
	$$
	{\rm supp}[\mathcal{V}_{t,x}] \cap \{ \vr=0,\ \vm\neq 0\} = \emptyset,
	$$
	see \cite[Remark 2.7]{BreFei17}. Hence DMV solutions allow for treating vacuum. 
	
	\subsection{Generating a DMV solution}
	\label{SSDMV}

	Suppose now that
	\begin{equation} \label{DMV3}
	\begin{split}
	&[ \vr_{0,n}, \vm_{0,n}, E_{0,n} ] \to [\vr_0, \vm_0, E_0] \quad\ \mbox{in}\ L^1(\Omega; R^{N+2}),\\
	&\vr_{0,n} > 0 \ \mbox{a.a. in}\ \Omega,\quad\ s(\vr_{0,n}, \vm_{0,n}, E_{0,n}) \geq s_0 > -\infty,
	\end{split}
	\end{equation}
	where $[\vr_{0,n}, \vm_{0,n}, E_{0,n}]$ are initial data of admissible weak solutions $[\vr_n, \vm_n, E_n]$ defined on $(0,T) \times \Omega$.
	
	Passing to a suitable subsequence, we may suppose that $\{ \vr_n, \vm_n, E_n \}_{n=1}^\infty$ generates a Young measure
	\[
	\left\{ \mathcal{V}_{t,x} \right\}_{(t,x) \in (0,T) \times \Omega)},\
	\mathcal{V}_{t,x} \in {\rm prob} [\mathcal{P}].
	\]
	As shown in \cite[Section 2.1]{BreFei17}, the Young measure $\left\{ \mathcal{V}_{t,x} \right\}_{(t,x) \in (0,T) \times \Omega)}$
	is a DMV solution of the Euler system in the sense of Definition \ref{DDMV1}. Moreover, in view of \eqref{DMV3}, the initial
	measure $\{ \mathcal{V}_{0,x} \}_{x \in \Omega}$ reads
	\[
	\mathcal{V}_{0,x} = \delta_{\vr_0(x), \vm_0(x), E_{0}(x)} \ \quad\mbox{for a.a.}\ x \in \Omega.
	\]

	\section{Application of the weak--strong uniqueness principle}
	\label{WS}
	
	The weak-strong uniqueness principle (Theorem 3.3 in \cite{BreFei17}) asserts that a DMV solution coincides
	with the strong solution emanating from the same initial data at least on the life span of the latter. Evoking the situation
	from Section \ref{SSDMV}, we suppose now that
	\[
	[\vr_{0,n}, \vm_{0,n}, E_{0,n}] \stackrel{(W)}{\rightarrow} [\vr_0, \vm_0, E_0],
	\]
	where $[\vr_0, \vm_0, E_0]$ are now regular initial data generating a $C^1-$solution $[\vr, \vm, E]$ of the Euler system
	(\ref{E1}--\ref{E3}) on a maximal time interval $[0, T_{{\rm max}})$. Without loss of generality, we may suppose that
	$0 < T < T_{{\rm max}}$.
	
	Applying the weak strong uniqueness principle we obtain
	\begin{equation} \label{DMV4}
	\mathcal{V}_{t,x} = \delta_{\vr(t,x), \vm(t,x), E(t,x)} \ \quad\mbox{for a.a.}\ (t,x) \in (0,T) \times \Omega.
	\end{equation}
	In particular, we may assume
	\[
	\vr_n \to \vr, \ \vm_n \to \vm,\ E_n \to E \ \quad\mbox{a.a. in}\ (0,T) \times \Omega
	\]
	for the associated weak solutions $[\vr_n, \vm_n, E_n]$. Consequently, as $[\vr_n, \vm_n, E_n]$ satisfy \eqref{E10}, we get
	\begin{align*}
	\int_0^T \intO{ \left[ \left( \frac{\vc{m} \otimes \vc{m} }{\vr} \right) : \Grad^2 \varphi - \frac{1}{N} \frac{ |\vc{m}|^2 }{\vr} \Del \varphi \right]
	} \dt = 0 \ \\\mbox{for any} \ \varphi \in \DC((0,T) \times \Omega).
	\end{align*}
	Finally, as the limit solution is continuous up to the time $t = 0$, we may infer that
	in particular,
	\begin{equation} \label{DMV5}
	\intO{ \left[ \left( \frac{\vc{m}_0 \otimes \vc{m}_0 }{\vr_0} \right) : \Grad^2 \varphi - \frac{1}{N} \frac{ |\vc{m}_0|^2 }{\vr_0} \Del \varphi \right]
	}  = 0 \quad\ \mbox{for any} \ \varphi \in \DC(\Omega).
	\end{equation}
	For $C^2$ initial data, relation \eqref{DMV5} can be rewritten as a non--linear differential equation
	\[
	\Div \Div \left( \frac{\vc{m}_0 \otimes \vc{m}_0 }{\vr_0} \right) - \Del \left( \frac{1}{N} \frac{ |\vc{m}_0|^2 }{\vr_0} \right) = 0 \quad\ \mbox{in}\ \Omega.
	\]

	\section{Reachability}
	\label{FI}
	
	In this section we finish the proof of Theorem \ref{TE1}. Introducing $\vc{w} = \frac{ \vc{m}_0 }{\sqrt{\vr_0}}$ we may write
	\eqref{DMV5} in a concise form
	\begin{equation} \label{FI1bis}
	\intO{ \left[ \vc{w} \otimes \vc{w}  : \Grad^2 \varphi - \frac{1}{N} |\vc{w}|^2 \Del \varphi \right]
	}  = 0 \quad\ \mbox{for any} \ \varphi \in \DC(\Omega).
	\end{equation}
	We consider solutions of \eqref{FI1bis} in the space $L^2(\Omega; R^N)$,
	\[
	\mathcal{S} = \left\{ \vc{w} \in L^2(\Omega; R^N)\ \Big| \ \vc{w} \ \mbox{solves}\ \eqref{FI1bis} \right\}.
	\]
	It is easy to check that:
	\begin{itemize}
		
		\item the set $\mathcal{S}$ is closed in $L^2(\Omega; R^N)$; 
		\item if $\vc{w} \in \mathcal{S}$, then $\lambda \vc{w} \in \mathcal{S}$ for any $\lambda \in R$.
		
	\end{itemize}
	
	\begin{lemma} \label{lemma41}
		The set $\mathcal{S}$ is nowhere dense in $L^2(\Omega; R^N)$, meaning the (closure of the) set $\mathcal{S}$ does not contain any ball in $L^2(\Omega; R^N)$.
	\end{lemma} 
	
	\begin{proof}
		Arguing by contradiction, we suppose that there is $\vc{w}_0 \in L^2(\Omega; R^N)$ and $r > 0$ such that the ball centered at $\vc{w}_0$ with radius $r$ is contained in $\mathcal{S}$. Hence
		\begin{equation} \label{CON1}
		\begin{split}
		\intO{ \left[ (\vc{w} + \vc{w}_0) \otimes (\vc{w} + \vc{w}_0)  : \Grad^2 \varphi - \frac{1}{N} |\vc{w} + \vc{w}_0|^2 \Del \varphi \right]
		}  = 0 \ \\\mbox{for any} \ \varphi \in \DC(\Omega), 
		\mbox{and any}\ \vc{w},\ \| \vc{w} \|_{L^2(\Omega; R^N)} \leq r.
		\end{split}
		\end{equation} 
		
		Next, we show that this implies that the ball centered at $0$ with radius $r$ is contained in $\mathcal{S}$, too. To this end, we write
		\[
		\begin{split}
		&\intO{ \left[ (\vc{w} + \vc{w}_0) \otimes (\vc{w} + \vc{w}_0)  : \Grad^2 \varphi - \frac{1}{N} |\vc{w} + \vc{w}_0|^2 \Del \varphi \right]
		}\\
		&= \intO{ \left[ \vc{w}  \otimes \vc{w}   : \Grad^2 \varphi - \frac{1}{N} |\vc{w}|^2 \Del \varphi \right] } \\
		&\quad+
		\intO{ \left[ (\vc{w} \otimes \vc{w}_0 + \vc{w}_0 \otimes \vc{w})  : \Grad^2 \varphi - \frac{2}{N} \vc{w} \cdot \vc{w}_0 \Del \varphi \right] }\\
		&\quad+\intO{ \left[ \vc{w}_0  \otimes \vc{w}_0   : \Grad^2 \varphi - \frac{1}{N} |\vc{w}_0|^2 \Del \varphi \right] }.
		\end{split}
		\] 
		Thus using the fact that $\vc{w}_0$ solves \eqref{FI1bis} and \eqref{CON1} we conclude
		\begin{equation} \label{CON2}
		\begin{split}
		&\intO{ \left[ \vc{w}  \otimes \vc{w}   : \Grad^2 \varphi - \frac{1}{N} |\vc{w}|^2 \Del \varphi \right] } \\&+
		\intO{ \left[ (\vc{w} \otimes \vc{w}_0 + \vc{w}_0 \otimes \vc{w})  : \Grad^2 \varphi - \frac{2}{N} \vc{w} \cdot \vc{w}_0 \Del \varphi \right] } \ =\ 0\\
		&\qquad\qquad\qquad\mbox{for any} \ \varphi \in \DC(\Omega)\ \mbox{and any}\ \vc{w},\ \| \vc{w} \|_{L^2(\Omega; R^N)} \leq r.
		\end{split}
		\end{equation}
		
		If relation \eqref{CON2} holds for any $\vc{w}$, $\| \vc{w} \|_{L^2(\Omega; R^N)} \leq r$, it must hold for $\lambda \vc{w}$, $0 \leq \lambda \leq 1$. Consequently,
		we get from \eqref{CON2},
		\[
		\begin{split}
		0 &= (\lambda^2 - \lambda)\intO{ \left[ \vc{w}  \otimes \vc{w}   : \Grad^2 \varphi - \frac{1}{N} |\vc{w}|^2 \Del \varphi \right] } \\
		&\quad+ \lambda \intO{ \left[ \vc{w}  \otimes \vc{w}   : \Grad^2 \varphi - \frac{1}{N} |\vc{w}|^2 \Del \varphi \right] } \\
		&\quad+ \lambda
		\intO{ \left[ (\vc{w} \otimes \vc{w}_0 + \vc{w}_0 \otimes \vc{w})  : \Grad^2 \varphi - \frac{2}{N} \vc{w} \cdot \vc{w}_0 \Del \varphi \right] } \\
		&= (\lambda^2 - \lambda)\intO{ \left[ \vc{w}  \otimes \vc{w}   : \Grad^2 \varphi - \frac{1}{N} |\vc{w}|^2 \Del \varphi \right] }
		\ \quad\mbox{for any}\ 0 \leq \lambda \leq 1.
		\end{split}
		\]
		We deduce
		\[
		\begin{split}
		&\intO{ \left[ \vc{w}  \otimes \vc{w}   : \Grad^2 \varphi - \frac{1}{N} |\vc{w}|^2 \Del \varphi \right]
		}  = 0 \ \\&\quad \mbox{for any} \ \varphi \in \DC(\Omega), 
		\mbox{and any}\ \vc{w},\ \| \vc{w} \|_{L^2(\Omega; R^N)} \leq r.
		\end{split}
		\]
		Thus the set $\mathcal{S}$ contains the ball of radius $r > 0$ centered at $0$ but as it is invariant with respect to multiplication by any real constant, we conclude 
		\[
		\mathcal{S} = L^2(\Omega; R^N),
		\]
		which is obviously false as soon as $N \geq 2$. 
		
	\end{proof}
	
	\begin{lemma} \label{lemma42}
		The set of reachable data in the sense of Definition \ref{ND2} is closed in $L^1((0,T) \times \Omega)$. 
	\end{lemma}

	\begin{proof}
		Let $[\vr_{0,n}, \vm_{0,n}, E_{0,n}]$ be a sequence of reachable data such that
		\[
		[\vr_{0,n}, \vm_{0,n}, E_{0,n}] \to 
		[\vr_0, \vm_0, E_0] \ \quad\mbox{in}\ L^1((0,T) \times \Omega; R^{N + 2}).
		\] 
		To see that the limit is reachable, we have to find a generating sequence satisfying \eqref{E9}, \eqref{E10}. 
		To this end, consider the generating sequences of data $[\vr_{0,n,m}, \vm_{0,n,m}, E_{0,n,m}]$ satisfying
		\[
		[\vr_{0,n,m}, \vm_{0,n,m}, E_{0,n,m}] \stackrel{(W)}{\rightarrow} [\vr_{0,n}, \vm_{0,n}, E_{0,n}]
		\ \quad\mbox{as}\ m \to \infty,
		\]
		meaning, in particular,
		\[
		\begin{split}
		\int_0^T \intO{ \left( \frac{\vm_{n,m} \otimes \vm_{n,m}} {\vr_{n,m}} - \frac{1}{N} \frac{ |\vm_{n,m}|^2}{\vr_{n,m}} \mathbb{I} \right)
			: \Grad^2 \varphi } \dt \to 0 \quad\ \mbox{as}\ m \to \infty\ \\\mbox{for any}\ \varphi \in \DC((0,T) \times \Omega)
		\end{split}
		\]
		for the corresponding solutions $\vr_{n,m}$, $\vm_{n,m}$ and any fixed $n$. Now, let
		\[
		\left\{ \varphi^k \right\}_{k = 1}^\infty, \ \varphi^k \in \DC((0,T) \times \Omega)
		\] 
		be a countable dense set in the Sobolev space
		\begin{equation} \label{AA1}
		\begin{split}
		W^{\ell,2}_0 ((0,T) \times \Omega) 
		\equiv \Ov{ \DC ((0,T) \times \Omega) }^{\| \cdot \|_{\ell,2} } \hookrightarrow C^2([0,T] \times \Ov{\Omega}) \ \\\mbox{as soon as} \ 
		\ell > \frac{N + 5}{2}.
		\end{split}
		\end{equation}
		
		By the diagonal method, we can find a subsequence
		\[
		\begin{split}
		[\vr_{0,(n, m)(j)}, \vm_{0, (n,m)(j)}, E_{0, (n,m)(j)} ] 
		\to [\vr_0, \vm_0, E_0] \ \quad\mbox{as}\ j \to \infty \ \\ \mbox{in}\ L^1((0,T) \times \Omega; R^{N+2})
		\end{split}
		\]
		such that
		\begin{equation} \label{AA2}
		\begin{split}
		\int_0^T \intO{ \left( \frac{\vm_{(n,m)(j)} \otimes \vm_{(n,m)(j)}} {\vr_{(n,m)(j)}} - \frac{1}{N} 
			\frac{ |\vm_{(n,m)(j)}|^2}{\vr_{(n,m)(j)}} \mathbb{I} \right)
			: \Grad^2 \varphi^k } \dt \to 0 \ \\\mbox{as}\ j \to \infty\ \mbox{for any}\ k.
		\end{split}
		\end{equation}
		Indeed, for given $j \geq 1$, we find $n(j)$ such that
		\[
		\left\| [\vr_{0,n(j)}, \vm_{0,n(j)}, E_{0,n(j)}] - [\vr_0, \vm_0, E_0] \right\| < \frac{1}{j}, 
		\] 
		and $m = m(n(j))$ such that
		\[
		\left\| [\vr_{0,n(j), m(n(j))}, \vm_{0,n(j), m(n(j))}, E_{0,n(j), m(n(j))}] - [\vr_{0,n(j)}, \vm_{0, n(j)}, E_{0, n(j)}] \right\| < \frac{1}{j},
		\]
		\[
		\begin{split}
		\left| \int_0^T \intO{ \left( \frac{\vm_{n(j),m(n(j))} \otimes \vm_{n(j),m(n(j))}} {\vr_{n(j),m(n(j))}} - \frac{1}{N} \frac{ 
				|\vm_{n(j),m(n(j))}|^2}{\vr_{n(j),m(n(j))}} \mathbb{I} \right)
			: \Grad^2 \varphi^k } \dt \right| < \frac{1}{j} \\ \mbox{for all}\ \varphi^k \in \DC((0,T) \times \Omega), \ k \leq j.
		\end{split}
		\]
		
		Finally, by virtue of \eqref{AA1}, any $\nabla^2_x \varphi$ can be uniformly approximated by $\nabla^2_x \varphi^k$, and, 
		as the energy $E_{n(j), m(j)}$ are bounded, we conclude that \eqref{AA2} holds for any $\varphi \in \DC((0,T) \times \Omega)$. 
		Thus we have shown that
		\[
		[\vr_{0,n(j), m(j)}, \vm_{0, n(j), m(j)}, E_{0, n(j), m(j)} ] 
		\stackrel{(W)}{\rightarrow}  [\vr_0, \vm_0, E_0] \ \quad\mbox{as}\ j \to \infty;
		\]
		whence $[\vr_0, \vm_0, E_0]$ is reachable.
		
	\end{proof}
	
	We are ready to complete the proof of Theorem \ref{TE1}. 
	The set of reachable data being closed (cf. Lemma \ref{lemma42}), its complement is open. In view of Lemma \ref{lemma41}, 
	we have only to show that the set of the smooth initial data $[\vr_0, \vm_0, E_0]$ giving rise to local-in-time regular solutions 
	is dense in the space $\texttt{L}^1_{+, s_0}(\Omega; R^{N+2})$. To this end, we record the following result by 
	Schochet \cite[Theorem 1]{Schochet86} that asserts the existence of smooth solutions whenever:
	
	\begin{itemize}
		\item $\Omega \subset R^N$ is a bounded domain with a sufficiently smooth boundary, say $\partial \Omega$ of class $C^\infty$;
		\item the initial data $[\vr_{0,E}, \vt_{0,E}, \vu_{0,E}]$ belong to the class
		\[
		\vr_{0, E}, \vt_{0, E} \in W^{m,2}(\Omega), \ \quad\vu_{0, E} \in W^{m,2}(\Omega; R^N),\quad\
		\vr_{0, E}, \ \vt_{0,E} > 0  \ \mbox{in}\ \Ov{\Omega}, 
		\]
		where $ m > N$;
		\item the compatibility conditions
		\[
		\partial^k_t \vu_{0,E} \cdot \vc{n}|_{\partial \Omega} = 0
		\]
		hold for $k=0,1,\dots,m$.
		
	\end{itemize}
	In particular, any trio of initial data, 
	\begin{equation}\label{AA3}
	\begin{split}
	&[\vr_0, \vm_0, E_0] \in C^\infty (\Ov{\Omega}; R^{N+2}), \ \\
	&\vm_0 = 0, \ \vr_0 = \Ov{\vr} > 0,\ 
	\vt_0 = \Ov{\vt} > 0\ \quad\mbox{on a neighborhood of}\ \partial \Omega,
	\end{split}
	\end{equation}
	gives rise to a smooth solution defined on a maximal time interval $[0, T_{{\rm max}})$. Seeing that the set 
	\eqref{AA3} is dense in $\texttt{L}^1_{+, s_0}(\Omega; R^{N+2})$, we have completed the proof of Theorem \ref{TE1}. 
	
	\begin{remark}
		Note that the result is basically independent of the choice of boundary conditions. Similarly, it can be easily extended to more general equation of state satisfying
		\eqref{E8}, including the isentropic case. Of course, the wild solution class is restricted only to those solutions that can be obtained
		by the splitting of the incompressible and acoustic part as specified in \eqref{NE7}.
	\end{remark}
	
	\section{Reachability for the Riemann problem} \label{DD}
	
	To the best of our knowledge there are only two methods available in the literature that give rise to infinitely many solutions of the compressible Euler system, which are both based on a splitting as in \eqref{NE6}. The first method \cite{FKKM17} was explained above, see \eqref{NE7}. The second method considers initial data constant in each of the two half spaces (Riemann data), see \cite{BKKMM18} and also \cite{ChiKre14}, \cite{ChiDelKre15}, \cite{KliMar18} for the isentropic case. Here -- instead of \eqref{NE7} -- equations $\eqref{NE6}_2$ and $\eqref{NE6}_3$ are replaced by 
	\begin{equation*} 
	\begin{split}
	\partial_t \vc{v} +\Div \left( \frac{\vc{m} \otimes \vc{m} }{\vr} - \frac{1}{N} \frac{|\vc{m}|^2}{\vr} \mathbb{I} - \vr 
	\mathbb{U} \right)   &= 0, \ \Big( \Div \vc{v} = 0 \Big), \\
	\partial_t (\Grad \Phi) + \Div(\vr \mathbb{U})
	+ \Grad \left( \frac{1}{N} \frac{|\vc{m}|^2}{\vr} \right) + \Grad p(\vr, \vc{m}, E)  &=0,\ \Big( \partial_t \vr + \Del \Phi = 0 \Big),
	\end{split}
	\end{equation*} 
	with a suitable matrix field $\mathbb{U}$. The density $\vr$, the acoustic potential $\Grad \Phi$ as well as the field $\mathbb{U}$ 
	are piecewise constant in $(0,T) \times \Omega$. In particular, there exists a partition of the set $(0,T) \times \Omega$ into a finite number 
	of sectors $Q_i, i=1, \dots m,$ such that the wild solutions satisfy 
	\[
	\begin{split}
	\int_{Q_i} \left[ \frac{\vc{m} \otimes \vc{m} }{\vr} - \frac{1}{N} \frac{ |\vm|^2 }{\vr} \mathbb{I} \right]
	: \Grad^2 \varphi \dxdt = 0 \ \\\mbox{for any} \ \varphi \in \DC(Q_i), \ i = 1,\dots, m.
	\end{split}
	\] 
	Consequently, we may accommodate the data reachable by this method replacing \eqref{E10} by
	\[
	\begin{split}
	\int_0^T \intO{ \left( \frac{\vm_n \otimes \vm_n} {\vr_n} - \frac{1}{N} \frac{ |\vm_n|^2}{\vr_n} \mathbb{I} \right)
		: \Grad^2 \varphi } \dt \to 0 \ \quad\mbox{as}\ n \to \infty\ \\\mbox{for any}\ \varphi \in \DC(Q),
	\end{split}
	\]
	where $Q \subset (0,T) \times \Omega$ is an open set such that $\Ov{Q} = [0,T] \times \Ov{\Omega}$.
	
	The above observation motivates the following extension of the concept of reachability. We say that $Q$ is 
	a \emph{partition} of the domain $(0,T) \times \Omega$ if 
	\begin{itemize}
		\item 
		$Q \subset (0,T) \times \Omega$ is an open set;
		\item $\Ov{Q} = [0,T] \times \Ov{\Omega}.$
	\end{itemize}
	We say that a family of partitions $\mathcal{Q}$ is a \emph{closed partition set} if it is closed with respect to the 
	Hausdorff complementary topology. More specifically, any sequence $\{ Q_n \}_{n=1}^\infty \subset \mathcal{Q}$ contains 
	a subsequence such that 
	\[
	Q^c_{n(k)} \stackrel{(H)}{\rightarrow} Q^c  \ \quad\mbox{as} \ k \to \infty \ \mbox{for some}\ Q \in \mathcal{Q},
	\]
	where the symbol $\stackrel{(H)}{\rightarrow}$ denotes convergence in the Hausdorff metric and $Q^c_n$ denotes the complement $Q^c_n \equiv ([0,T] \times \Ov{\Omega} ) \setminus Q_n$.
	
	An example of a closed partition set related to the convex integration method is 
	\[
	\begin{split}
	\mathcal{Q} = \Big\{ &Q \subset (0,T) \times \Omega \ \Big| \\
	&\ Q = ((0,T) \times \Omega) \setminus 
	(\cup_{i = 1}^M H_i), \ H_i - \mbox{a hyperplane in}\ R^{N + 1} \Big\},
	\end{split}
	\]
	where $M$ is a given positive integer. Note that this is indeed a closed partition set since $(0,T)\times\Omega$ is a bounded subset of $R^{N+1}$.
	
	The following property is well known, cf. \cite{Pironneau84}: 
	\begin{equation} \label{DD1}
	\begin{array}{c}
	Q^c_{n} \stackrel{(H)}{\rightarrow} Q^c \ \mbox{and}\ K \subset Q \ \\\mbox{is a compact set} 
	\end{array}
	\Longrightarrow 
	\begin{array}{c}
	\mbox{there exists}\ n(K) \ \mbox{such that}\ \\K \subset Q_n \ \mbox{for all}\ n \geq n(K)
	\end{array}.
	\end{equation}
	
	Next, we introduce the following generalization of $(W)-$convergence. Let $\mathcal{Q}$ be a closed partition set. 
	We say that a sequence of data $\left\{ \vr_{0,n}, \vm_{0,n}, E_{0,n} \right\}_{n=1}^\infty$ \emph{$(W[\mathcal{Q}])$-converges} to
	$[\vr_0, \vm_0, E_0]$,
	\[
	[ \vr_{0,n}, \vm_{0,n}, E_{0,n} ] \stackrel{(W[\mathcal{Q}])}{\rightarrow} [\vr_0, \vm_0, E_0],
	\]
	if
	\begin{itemize}
		\item
		\[
		\vr_{0,n} > 0,\quad \ s(\vr_{0,n}, \vm_{0,n}, E_{0,n}) \geq  s_0 > -\infty;
		\]
		\item
		\[
		[ \vr_{0,n}, \vm_{0,n}, E_{0,n} ] \to [\vr_0, \vm_0, E_0] \ \quad\mbox{in}\ L^1(\Omega; R^{N+2});
		\]
		\item
		the initial data $[\vr_{0,n},
		\vm_{0,n}, E_{0,n}]$ give rise to a sequence of admissible weak solutions $[\vr_n, \vm_n, E_n]$ satisfying
		\begin{equation} \label{DD2}
		\begin{split}
		\int_0^T \intO{ \left( \frac{\vm_n \otimes \vm_n} {\vr_n} - \frac{1}{N} \frac{ |\vm_n|^2}{\vr_n} \mathbb{I} \right)
			: \Grad^2 \varphi } \dt \to 0 \ \quad\mbox{as}\ n \to \infty\ \\\mbox{for any}\ \varphi \in \DC(Q)
		\end{split}
		\end{equation}
		for some $Q \in \mathcal{Q}$. 
		
	\end{itemize}
	
	\begin{definition} \label{DDD1}
		
		Let $\mathcal{Q}$ be a closed partition set.
		We say that a trio $[\vr_0, \vm_0, E_0]$ is \emph{$\mathcal{Q}-$reachable} if there exists a sequence of initial data 
		$\left\{ \vr_{0,n}, \vm_{0,n}, E_{0,n} \right\}_{n=1}^\infty$ such that
		\[
		[ \vr_{0,n}, \vm_{0,n}, E_{0,n} ] \stackrel{(W[\mathcal{Q}])}{\rightarrow}  [\vr_0, \vm_0, E_0].
		\]

	\end{definition}
	
	Obviously any reachable data in the sense of Definition \ref{ND2} are $\mathcal{Q}$-reachable so the set 
	of $\mathcal{Q}-$reachable data is always larger for any closed partition set.
	
	In order to adapt the arguments used in Sections \ref{WS}, \ref{FI} we need the following result.
	
	\begin{lemma} \label{LDD1}
		
		Let $\mathcal{Q}$ be a closed partition set. 
		
		Then the set of $\mathcal{Q}-$reachable data is closed in $L^1((0,T) \times \Omega; R^{N+2})$.

	\end{lemma}
	
	\begin{proof} 
		
		Let $[\vr_{0,n}, \vm_{0,n}, E_{0,n}]$ be a sequence of $\mathcal{Q}-$reachable data such that
		\[
		[\vr_{0,n}, \vm_{0,n}, E_{0,n}] \to 
		[\vr_0, \vm_0, E_0] \ \quad\mbox{in}\ L^1((0,T) \times \Omega; R^{N + 2}).
		\] 
		Let $[\vr_{0,n,m}, \vm_{0,n,m}, E_{0,n,m}]$ be the corresponding generating sequences,
		\[
		[\vr_{0,n,m}, \vm_{0,n,m}, E_{0,n,m}] \stackrel{(W[\mathcal{Q}])}{\rightarrow} [\vr_{0,n}, \vm_{0,n}, E_{0,n}]
		\ \quad\mbox{as}\ m \to \infty.
		\]
		In particular, the corresponding solutions $\vr_{m,n}$, $\vm_{m,n}$ satisfy
		\[
		\begin{split}
		\int_0^T \intO{ \left( \frac{\vm_{n,m} \otimes \vm_{n,m}} {\vr_{n,m}} - \frac{1}{N} \frac{ |\vm_{n,m}|^2}{\vr_{n,m}} \mathbb{I} \right)
			: \Grad^2 \varphi } \dt \to 0 \ \quad\mbox{as}\ m \to \infty\ \\\mbox{for any}\ \varphi \in \DC(Q_n)
		\end{split}
		\]
		for some $Q_n \in \mathcal{Q}$ and any fixed $n$. 
		
		As the partition set $\mathcal{Q}$ is closed, there exists a partition 
		$Q \in \mathcal{Q}$ such that 
		\begin{equation} \label{DD3}
		K \subset Q \ \mbox{a compact set} 
		\Rightarrow \mbox{there exist}\ n(K) \ \mbox{such that}\ K \subset Q_n \ \mbox{for all}\ n \geq n(K)
		\end{equation}
		at least for a suitable subsequence (not relabeled). 
		
		Let
		\[
		\left\{ \varphi^k \right\}_{k = 1}^\infty, \ \varphi \in \DC(Q)
		\] 
		be a countable dense subset of the Sobolev space
		\begin{equation} \label{DD4}
		W^{\ell,2}_0 (Q) 
		\equiv \Ov{ \DC (Q) }^{\| \cdot \|_{\ell,2} } \hookrightarrow C^2(\Ov{Q}) \quad\ \mbox{as soon as} \ 
		\ell > \frac{N + 5}{2}.
		\end{equation}
		Using property \eqref{DD3}, we can find a subsequence
		\[
		\begin{split}
		[\vr_{0,n(j), m(j)}, \vm_{0, n(j), m(j)}, E_{0, n(j), m(j)} ] 
		\to [\vr_0, \vm_0, E_0] \ \quad\mbox{as}\ j \to \infty \  \\\mbox{in}\ L^1((0,T) \times \Omega; R^{N+2})
		\end{split}
		\]
		satisfying
		\begin{equation} \label{DD5}
		\begin{split}
		\int_0^T \intO{ \left( \frac{\vm_{n(j),m(j)} \otimes \vm_{n(j),m(j)}} {\vr_{n(j),m(j)}} - \frac{1}{N} 
			\frac{ |\vm_{n(j),m(j)}|^2}{\vr_{n(j),m(j)}} \mathbb{I} \right)
			: \Grad^2 \varphi^k } \dt \to 0 \ \\\mbox{as}\ j \to \infty\ \mbox{for any}\ k.
		\end{split}
		\end{equation}
		
		By virtue of \eqref{DD4}, any $\nabla^2_x \varphi$ can be uniformly approximated by $\nabla^2_x \varphi^k$, and, 
		as the energies $E_{n(j), m(j)}$ are bounded, we conclude that \eqref{DD5} holds for any $\varphi \in \DC(Q)$. 
		We may infer that 
		\[
		[\vr_{0,n(j), m(j)}, \vm_{0, n(j), m(j)}, E_{0, n(j), m(j)} ] 
		\stackrel{(W[\mathcal{Q}])}{\rightarrow}  [\vr_0, \vm_0, E_0] \ \quad\mbox{as}\ j \to \infty;
		\]
		whence $[\vr_0, \vm_0, E_0]$ is $\mathcal{Q}-$reachable.

	\end{proof}
	
	Finally, we observe that any regular initial data $[\vr_0, \vm_0, E_0]$ satisfy \eqref{DMV5}. To see this, we first observe, 
	similarly to Section \ref{WS}, that 
	\[
	\int_0^T \intO{ \left[ \left( \frac{\vc{m} \otimes \vc{m} }{\vr} \right) : \Grad^2 \varphi - \frac{1}{N} \frac{ |\vc{m}|^2 }{\vr} \Del \varphi \right]
	} \dt = 0 \ \quad\mbox{for any} \ \varphi \in \DC(Q).
	\]
	As $\vr$, $\vm$ are smooth, we get 
	\[
	\Div \Div \left( \frac{\vc{m} \otimes \vc{m} }{\vr} \right) - \frac{1}{N} \Del \frac{ |\vc{m}|^2 }{\vr} = 0
	\ \quad\mbox{in}\ Q.
	\]
	However, as $\Ov{Q} = [0,T] \times \Ov{\Omega}$ and the second derivatives of $\rho$, $\vm$ are continuous, this implies
	\[ 
	\Div \Div \left( \frac{\vc{m} \otimes \vc{m} }{\vr} \right) - \frac{1}{N} \Del \frac{ |\vc{m}|^2 }{\vr} = 0
	\ \quad\mbox{in}\ (0,T) \times \Omega.
	\]
	In particular, we obtain \eqref{DMV5} for the initial data.

	We have shown the following extension of Theorem \ref{TE1}.
	
	\begin{theorem} \label{TDD1}
		Let $s_0 \in R$ be given and $\Omega \subset R^N$, $N=2,3$ be a bounded smooth domain.
		Let $\mathcal{Q}$ be a closed partition set in $(0,T) \times \Omega$. \\
		Then the complement of the set of $\mathcal{Q}-$reachable  data {is} an open dense set in
		$\texttt{L}^1_{+,s_0} (\Omega; R^{N + 2})$.
	\end{theorem}

	As already pointed out, Theorem \ref{TDD1} accommodates the limits of the wild initial data constructed via the Riemann problem with $M$ fans 
	in the sense of \cite{ChiDelKre15}.

	\begin{appendix}

		\section{Wild solutions}
		\label{A}

		The existence of weak solutions for the Euler system in higher space dimensions is an outstanding open problem. Nevertheless, the recently developed adaptation of the convex integration technique to the incompressible Euler system provided certain results also for
		the compressible case. We report the following result proved in \cite{FKKM17}.
		
		\begin{theorem} \label{TS1}
			
			Let the initial data $\vr_0$ and $\vt_0$ be piecewise constant strictly positive functions defined on $\Omega$. And let $T > 0$ be given.
			
			Then there exists a constant $\Lambda_0 > 0$ such that for any $\Lambda \geq \Lambda_0$ there is $\vu_0 \in L^\infty(\Omega; R^N)$
			such that the problem \eqref{E1}, \eqref{E5}, \eqref{E8} admits infinitely many weak solutions in $(0,T) \times \Omega$, with the initial data
			\[
			\vr_0, \quad\ \vc{m}_0 = \vr_0 \vu_0, \quad\ E_0 = \frac{1}{2} \vr_0 |\vu_0|^2 + \vr_0 e(\vr_0, \vt_0).
			\]
			
			In addition, these solutions enjoy the following properties:
			\begin{equation} \label{S1}
			- \frac{1}{N} \frac{|\vm|^2}{\vr} = - \Lambda + p \ \quad\mbox{for a.a.}\ (t,x)
			\ \mbox{including}\ t = 0, 
			\end{equation}
			\begin{equation} \label{S2}
			\Div \vc{m} = 0,
			\end{equation}
			\begin{equation} \label{S3}
			\vr(t,x) = \vr_0 (x) \ \quad\mbox{for all}\ t \geq 0,
			\end{equation}
			\begin{equation} \label{S4}
			S(t,x) = S_0(x) \quad\ \mbox{for all}\ t \geq 0,
			\end{equation}
			and the renormalized entropy inequality \eqref{RenEntr} holds as equality.
			
		\end{theorem}
		
		Note that here \eqref{NE7} is satisfied with $\Phi = 0$.

	\end{appendix}

	\section*{Acknowledgements}
		The research of E. F.~leading to these results has received funding from the
		Czech Sciences Foundation (GA\v CR), Grant Agreement
		18--05974S. The Institute of Mathematics of the Academy of Sciences of
		the Czech Republic is supported by RVO:67985840.


\begin{thebibliography}{}
		%
		%
		\bibitem{BKKMM18}
		Al~Baba, H., Klingenberg, C., Kreml, O., M\'acha, V., Markfelder, S.: Nonuniqueness of admissible weak solutions to the {R}iemann problem for the full {E}uler system in two dimensions. \textit{SIAM J. Math. Anal.} \textbf{52}(2), 1729--1760, 2020
		
		\bibitem{BreFei17}
		B{\v r}ezina, J., Feireisl, E.: Measure-valued solutions to the complete {E}uler system. \textit{J. Math. Soc. Japan} \textbf{70}(4), 1227--1245, 2018
		
		\bibitem{CheFri02}
		Chen, G.-Q., Frid, H., Li, Y.: Uniqueness and stability of {R}iemann solutions with large oscillation in gas dynamics. \textit{Comm. Math. Phys.} \textbf{228}(2), 201--217, 2002
		
		\bibitem{Chiodaroli14}
		Chiodaroli, E.: A counterexample to well-posedness of entropy solutions to the compressible {E}uler system. \textit{J. Hyperbolic Differ. Equ.} \textbf{11}(3), 493--519, 2014
		
		\bibitem{ChiDelKre15}
		Chiodaroli, E., De~Lellis, C., Kreml, O.: Global ill-posedness of the isentropic system of gas dynamics. \textit{Comm. Pure Appl. Math.} \textbf{68}(7), 1157--1190, 2015
		
		\bibitem{ChiKre14}
		Chiodaroli, E., Kreml, O.: On the energy dissipation rate of solutions to the compressible isentropic {E}uler system. \textit{Arch. Ration. Mech. Anal.} \textbf{214}(3), 1019--1049, 2014
		
		\bibitem{DelSze10}
		De~Lellis, C., Sz{\'e}kelyhidi~Jr., L.: On admissibility criteria for weak solutions of the {E}uler equations. \textit{Arch. Ration. Mech. Anal.} \textbf{195}(1), 225--260, 2010
		
		\bibitem{Feireisl16}
		Feireisl, E.: \textit{Weak solutions to problems involving inviscid fluids.} In: ``Mathematical Fluid Dynamics, Present and Future'', pp. 377--399, Springer Proceedings in Mathematics and Statistics \textbf{183}, Springer-Verlag, Tokyo, 2016
		
		\bibitem{FKKM17}
		Feireisl, E., Klingenberg, C., Kreml, O., Markfelder, S.: On oscillatory solutions to the complete Euler system. \textit{J. Differential Equations} \textbf{269}(2), 1521--1543, 2020
		
		\bibitem{KliMar18}
		Klingenberg, C., Markfelder, S.: The {R}iemann problem for the multidimensional isentropic system of gas dynamics is ill-posed if it contains a shock. \textit{Arch. Ration. Mech. Anal.} \textbf{227}(3), 967--994, 2018
		
		\bibitem{Pironneau84}
		Pironneau, O.: \textit{Optimal shape design for elliptic systems.} Springer series in {C}omputational {P}hysics, Springer-Verlag, New York, 1984
		
		\bibitem{Schochet86}
		Schochet, S.: The compressible {E}uler equations in a bounded domain: existence of solutions and the incompressible limit. \textit{Comm. Math. Phys.} \textbf{104}(1), 49--75, 1986
		
		\bibitem{SzeWie12}
		Sz{\'e}kelyhidi~Jr., L., Wiedemann, E.: Young measures generated by ideal incompressible fluid flows. \textit{Arch. Ration. Mech. Anal.} \textbf{206}(1), 333--366, 2012
	\end{thebibliography}
	

\end{document}